\documentclass[reqno,A4paper]{amsart}
\usepackage{amsmath,amssymb,amscd,amsthm,verbatim,alltt,amsfonts,array}
\usepackage{mathrsfs}
\usepackage[english]{babel}
\usepackage{latexsym}
\usepackage{amssymb}
\usepackage{euscript}
\usepackage{graphicx}
\usepackage{csquotes}

\usepackage{texdraw}
\usepackage{hyperref}
\usepackage[alphabetic]{amsrefs}
\usepackage[alphabetic]{amsrefs}
\usepackage[utf8]{inputenc}   
\usepackage[T1]{fontenc}
\usepackage[utf8]{inputenc}
\usepackage[T1]{fontenc}
\newtheorem{thm}{Theorem}[section]
\newtheorem{cor}[thm]{Corollary}
\newtheorem{prop}[thm]{Proposition}

\newtheorem{lem}[thm]{Lemma}

\newtheorem{ex}{Example}[section]
\newcommand{\be}{\begin{equation}}
\newcommand{\ee}{\end{equation}}
\newcommand{\ben}{\begin{enumerate}}
\newcommand{\een}{\end{enumerate}}
\newcommand{\beq}{\begin{eqnarray}}
\newcommand{\eeq}{\end{eqnarray}}
\newcommand{\beqn}{\begin{eqnarray*}}
\newcommand{\eeqn}{\end{eqnarray*}}

\author{Nasrin Sadeghzadeh}
\address{Department of Mathematics\\
University of Qom\\
Alghadir Bld\\
Qom\\
Iran.\\
Postal code:3716146611}
\email{nsadeghzadeh@qom.ac.ir}

\author{Najmeh Sajjadi Moghadam}
\address{Department of Mathematics\\
University of Qom\\
Alghadir Bld\\
Qom\\
Iran.\\
Postal code: 3716146611}
\email{nsajjadi1398@gmail.com}
\title[Beyond the Generalized Douglas Weyl spaces]
      {Beyond the Generalized Douglas Weyl spaces}

\begin{document}
\begin{center}
{\small
\textbf{Preprint Notice:} This is the original submitted version (preprint) of the article \\
``Beyond the generalized Douglas–Weyl spaces'' \\
by Nasrin Sadeghzadeh and Najmeh Sajjadi Moghadam. \\[2mm]
The final peer-reviewed version is published in \\
\textit{Publications Mathematicae Debrecen}, Volume 106, Issues 3-4, 2025, pp. 405–417. \\
DOI: \href{https://doi.org/10.5486/PMD.2025.9961}{10.5486/PMD.2025.9961}. \\[1mm]
Please cite the published version when referencing this work.
}
\end{center}
\maketitle
\begin{abstract}
This paper gives new insights into the class of Generalized Douglas Weyl ($GDW$)-metrics. This projective invariant class of Finsler metrics, contains some well-known Finsler metrics such as Douglas, Weyl and $R$-quadratic metrics. Here, some new sub-classes of $GDW$-metrics are constructed and considered as the explicit Finsler metrics. Many illustrative and interesting examples are presented.
\end{abstract}


\section{Introduction}
Two regular metrics on a manifold $M$ are called projectively related if they have the same geodesics as the point sets. In Physics, a geodesic represents the equation of motion that determines the phenomena of the space. A geodesic curve $c(t)$ in a Finsler space $(M, F)$ is defined by the second order system of differential equations
\be
\frac{d^2 c^i}{dt^2} + 2 G^i (c(t), \dot{c}(t))=0,
\ee
where $G^i$ are local functions on $TM$ given by
\[
G^i(x, y):= \frac{1}{4} g^{il}(x,y) \{ \frac{\partial^2 F^2}{\partial x^k \partial y^l} y^k - \frac{\partial F^2}{\partial x^l}\},\quad for \quad y\in T_xM.
\]
Projective Finsler geometry studies equivalent Finsler metrics on the same manifold with the same geodesics as points. J. Douglas introduces the (projective) Douglas curvature and the (projective) Weyl curvature in \cite{Douglas}. Douglas metrics are Finsler metrics with vanishing Douglas curvature, $D_j{^i}_{kl}=0$. Somewhere around, a Douglas metric is a Finsler metric which is locally projectively equivalent to a Riemannian metric. The class of Douglas metrics contains all Riemannian metrics and the locally projectively flat Finsler metrics. But, there are many Douglas metrics which are not Riemannian.
Moreover, it should be mentioned that there are several Douglas metrics which lack local projective flatness; Indeed, our findings demonstrate
\be\label{LPFinD}
\{Locally\hspace{1mm} Projectively\hspace{1mm}Flat\}\hspace{1mm} \subsetneqq \hspace{1mm}\{Douglas\hspace{1mm} metrics\}.
\ee
There is another well-known projective invariant class of Finsler metrics namely the class of Generalized Douglas-Weyl ($GDW$)-metrics. The Finsler metrics satisfying,
\[
D_j{^i}_{kl|m}y^m=T_{jkl}y^i,
\]
for some tensor $T_{jkl}$, where $D_j{^i}_{kl|m}$ denotes the horizontal covariant derivatives of $D_j{^i}_{kl}$ with respect to the Berwald connection of $F$, are called $GDW$-metrics \cite{Bac}.
While all Douglas metrics are considered $GDW$-metrics, it has been demonstrated that there are specific $GDW$-metrics that do not adhere to the Douglas type. In other words, we show that
\be\label{DinGDW}
\{Locally\hspace{1mm} Projectively\hspace{1mm}Flat\}\hspace{1mm} \subsetneqq \hspace{1mm}\{Douglas\hspace{1mm} metrics\}\hspace{1mm} \subsetneqq \hspace{1mm} \{{GDW-metrics}\}.
\ee
Based on Douglas curvature, a new class of Finsler metrics so called $\bar{D}$-metrics is introduced which includes all the Douglas metrics. A Finsler metric $F$ is called $\bar{D}$-metric if ${D}_j{^i}_{kl|m}-{D}_j{^i}_{km|l}=0$ or equivalently $D_j{^i}_{kl|m}y^m=0$.
It is evident that all Douglas metrics are contained in $\bar{D}$-metrics, but there are also numerous $\bar{D}$-metrics that are not Douglas, as illustrated within this study. It means that
\be\label{DinbarD}
\{Locally\hspace{1mm} Projectively\hspace{1mm} Flat\}\hspace{1mm} \subsetneqq \hspace{1mm}\{Douglas\hspace{1mm} metrics\}\hspace{1mm} \subsetneqq \hspace{1mm} \{\bar{D}-metrics\}.
\ee
It has been proven that the category of $\bar{D}$-metrics is completely contained within the set of $GDW$-metrics, signifying an interesting distinction between them: there are remarkable $GDW$-metrics which do not qualify as $\bar{D}$-metric. That is
\[
\{Locally\hspace{1mm} Projectively\hspace{1mm} Flat\}\hspace{1mm} \subsetneqq \hspace{1mm}\{Douglas\hspace{1mm} metrics\}\hspace{1mm}
\]
\[
\subsetneqq \hspace{1mm} \{\bar{D}-metrics\}\hspace{1mm} \subsetneqq \hspace{1mm} \{GDW-metrics\}.
\]
There are other interesting classes of Finsler metrics which are the subset of the class of $GDW$-metrics. $R$-quadratic and $PR$-quadratic Finsler metrics are some great examples of them. R-quadratic Finsler spaces form a rich class of Finsler spaces which were first introduced by B\'{a}cs\'{o} and Matsumoto and could be considered as a generalization of Berwald metrics \cite{BacMat}. A Finsler metric $(M,F)$ is called R-quadratic if its Riemann curvature $R_y$ is quadratic in $y\in T_{x}M$. In \cite{Rq}, it is proved that every R-quadratic Finsler metric is a $GDW$-metrics.
\\
For a Finsler metric $(M,F)$, the Riemann curvature of a projective spray, which introduced in \cite{Sh2}, is called Projective Riemann curvature ($PR$). The Projective Ricci curvature is defined as the Ricci curvature of the projective spray, too. A Finsler metric $(M,F)$ is called $PR$-quadratic if its Projective Riemann curvature $PR_y$ is quadratic in $y \in T_{x}M$.\\
All the above important classes of Finsler metrics belong to the class of $GDW$-metrics. In this paper, we consider these luminous metrics and find the interesting relations between the subclasses of $GDW$-metrics. In fact, a better and more in-depth understanding of $GDW$-metrics  and their sub-classes is provided.

\section{Preliminaries}
A  Finsler metric  on a manifold $M$ is  a non-negative function $F$ on $TM$ having
the following properties
\ben
\item[(a)] $F$ is
 $C^{\infty}$ on $TM\setminus \{0\}$,
\item[(b)]
$F(\lambda y) =\lambda F(y)$, $\forall \lambda >0$, $\ y\in TM$,
\item[(c)] For each $y\in T_xM$,
the following quadratic form ${\bf g}_y$ on $T_xM$ is positive definite,
\be
{\bf g}_y(u, v):= {1\over 2} \Big [ F^2(y+ s u + tv ) \Big ]\Big |_{s, t=0}, \ \ \ \ \ \ u, v\in T_xM.
\ee
\een
At each point $x\in M$, $F_x:= F|_{T_xM}$ is an Euclidean norm
if and only if ${\bf g}_y$ is independent of $y\in T_xM\setminus\{0\}$. \\
To measure the non-Euclidean feature of $F_x$, define
${\bf C}_y: T_xM \times T_xM \times T_xM \to R$ by
\be
{\bf C}_y(u, v, w):=
{1\over 2} {d \over dt} \Big [ {\bf g}_{y+tw} (u, v) \Big ] \Big |_{t=0}, \ \ \ \ \ \ u, v, w\in T_xM.
\ee
 The family ${\bf C}:=\{{\bf C}_y\}_{y\in TM\setminus \{0\} }$ is called the {\it Cartan torsion}. A curve $c(t)$ is called a {\it geodesic} if it satisfies
\be
{d^2 c^i\over dt^2} + 2 G^i (c(t), \dot{c}(t))=0,
\ee
where $G^i$ are local functions on $TM$ given by
\be
G^i(x, y):= \frac{1}{4} g^{il} \{ \frac{\partial^2 F^2}{\partial x^k \partial y^l} y^k - \frac{\partial F^2}{\partial x^l}\},\quad y\in T_xM, \label{Gi}
\ee
and called the spray coefficients of $F$.
The Riemann curvature $R_y = R{^i}_k \frac{\partial}{\partial x^i}\bigotimes dx^k$ of $F$ is given by
\[
R{^i}_k=2\frac{\partial G^i}{\partial x^k}-\frac{\partial^2 G^i}{\partial x^m \partial y^k}y^m+ 2G^m\frac{\partial^2 G^i}{\partial y^m \partial y^k}-\frac{\partial G^i}{\partial y^m}\frac{\partial G^m}{\partial y^k}.
\]
For the Riemann curvature of Finsler metric $F$ one has \cite{Sh2}
\be\label{Rikl}
R{^i}_{kl}=\frac{1}{3}(R{^i}_{k.l}-R{^i}_{l.k}), \quad and \quad R_j{^i}_{kl}=R{^i}_{kl.j}.
\ee
Here, $"{}_{.k}"$ denotes the differential with respect to $y^k$.\\
$F$ is called a Berwald metric if $G^i(y)$ are quadratic in $y\in T_xM$ for all $x\in M$.
For $y\in T_xM_0$, define
\[
B_y:T_xM \times T_xM \times T_xM\rightarrow T_xM
\]
\[
B_y(u,v,w)=B_j{^i}_{kl}u^j v^k w^l \frac{\partial}{\partial x^i},
\]
where
$
B_j{^i}_{kl}=\frac{\partial^3 G^i}{\partial y^j \partial y^k \partial y^l}.
$
Put
\[
E_y:T_xM \times T_xM \rightarrow \mathbb{R}
\]
\[
E_y(u,v)=E_{jk} u^j v^k,
\]
where $E_{jk}=\frac{1}{2}B_j{^m}_{km}$, $u=u^i\frac{\partial}{\partial x^i}$,
$v=v^i \frac{\partial}{\partial x^i}$ and $w=w^i\frac{\partial}{\partial x^i}$.
$B$ and $E$ are called the Berwald curvature and mean Berwald curvature, respectively. $F$ is called a Berwald metric and Weakly Berwald (WB) metric if $B=0$ and $E=0$, respectively \cite{Sh3}.
The $S$-curvature $S(x,y)$ was introduced as follows \cite{Sh3}
\[
S(x,y)=\frac{d}{dt}[\tau(\gamma(t),\gamma'(t)]_{|t=0},
\]
where $\tau(x, y)$ is the distortion of the metric $F$ and $\gamma(t)$ is the geodesic with $\gamma(0)=x$ and $\gamma'(0)=y$ on $M$. It is considerable that \cite{Sh2}
\be\label{ES}
E_{ij}=\frac{1}{2}S_{.i.j},
\ee
where $.i$ denotes the differential with respect to $y^i$, $\frac{\partial}{\partial y^i}$.
Let
\be\label{Douglas}
D_j{^i}_{kl}=B_j{^i}_{kl}-\frac{1}{n+1}\frac{\partial^3}{\partial y^j \partial y^k \partial y^l}(\frac{\partial G^m}{\partial y^m}y^i).
\ee
It is easy to verify that $D:=D_j{^i}_{kl} dx^j\otimes \frac{\partial}{\partial x^i}\otimes dx^k \otimes dx^l$ is a well-defined tensor on slit tangent bundle $TM_0$.
We call $D$ the Douglas tensor. The Douglas tensor $D$ is a non-Riemannian projective invariant, namely,
if two Finsler metrics $F$ and $\bar{F}$ are projectively equivalent,
\[
G^i=\bar{G^i}+P y^i,
\]
where projective factor $P=P(x,y)$ is positively $y$-homogeneous of degree one, then the Douglas tensor of $F$ is the same as that of $\bar{F}$ \cite{DShen}, \cite{Sh2}.
One could easily show that
\be\label{D2}
D_j{^i}_{kl}=B_j{^i}_{kl}-\frac{2}{n+1}\{E_{jk}\delta^i_l+E_{jl}\delta^i_k+E_{kl}\delta^i_j+E_{jk.l}y^i\}.
\ee
Douglas curvature, ${D_j}^i{}_{kl}$, is a projective invariant constructed from the Berwald curvature.
Finsler metrics with ${D_j}^i{}_{kl}=0$ are called Douglas metrics.
A Finsler metric is called $\bar{D}$-metric if $\bar{D}_j{^i}_{klm}=0$, where
\be
\bar{D}_j{^i}_{klm}=D_j{^i}_{kl|m}-D_j{^i}_{km|l}.
\ee
Clearly, this class of metrics includes all Douglas metrics. However, as the examples presented in previous section, there are lots of $\bar{D}$-metrics which are not of Douglas type (non-trivial $\bar{D}$-metrics).
The metrics with the following condition are called GDW metric which are projective invariant.
\[
{D_j}^i{}_{kl|m}y^m=T_{jkl}y^i,
\]
for some tensors $T_{jkl}$, where ${D_j}^i{}_{kl|m}$ denotes the horizontal derivatives of ${D_j}^i{}_{kl}$ with respect to the Berwald connection of $F$.\\
\begin{lem} \cite{Sh2}
Let $F$ and $\bar{F}$ be two projectively equivalent Finsler metrics on $M$. The Riemann curvatures are related by
\be\label{Rieproj}
\bar{R}{^i}_k=R{^i}_k+E\delta{^i}_k+\tau_k y^i,
\ee
where
\[
E=P^2-P_{|m}y^m, \quad \quad \tau_k=3(P_{|k}-PP_{.k})+E_{.k}.
\]
Here $P_{|k}$ denotes the covariant derivative of projective factor $P$ with respect to $\bar{F}$.\\
\end{lem}
For a spray $G$ on an $n$-dimensional manifold $M$ and given a volume form $dV$ on $M$, we can construct a new spray by
\[
\tilde{G}:=G+\frac{2S}{n+1}Y.
\]
The spray $\tilde{G}$ is called the projective spray of $(G, dV)$. In local coordinates,
\be\label{PSpraylocal}
\widetilde{G}^i=G^i-\frac{S}{n+1}y^i.
\ee
The projective Ricci curvature of $(G, dV)$ is defined as the Ricci curvature of $\tilde{G}$, namely,
\[
PRic_{(G, dV)}:=Ric_{\tilde{G}}.
\]
Then by a simple computation one has
\be\label{PRic}
PRic_{(G,dV)}=Ric+(n-1)\{\frac{S_{|0}}{n+1}+[\frac{S}{n+1}]^2\}.
\ee
where $Ric = Ric_G$ is the Ricci curvature of the spray $G$, $S=S_{(G,dV)}$ is the S-curvature of $(G,dV)$ and $S_{|0}$ is the covariant derivative of $S$  along a geodesic of $G$. It is known that $\tilde{G}$ remains unchanged under a projective change of $G$ with $dV$ fixed, thus $PRic_{(G,dV)}=Ric_{\tilde{G}}$ is a projective invariant of $(G,dV)$. For a Finsler metric $(M,F)$, the Riemann curvature of a projective spray is called projective Riemann curvature,
\[
{PR{^i}_k}_{(G,dV)}= {R{^i}_k}_{\widetilde{G}}.
\]
A Finsler metric $(M,F)$ is called $PR$-quadratic Finsler metric if $PR_j{^i}_{kl.m}=0$.
\section{Generalized Douglas Weyl Finsler metrics}
Douglas tensor is an important projective invariant in Finsler spaces. This significant tensor has been taken into consideration in the subsequent section.\\
Based on Douglas curvature, the class of $GDW$-metrics ($GDW(M)$) has been introduced which defines another well-known projective invariant Finsler metrics \cite{GDWTa}. The subsequent study focuses on the classes of Finsler metrics that are proper subsets of $GDW(M)$. As a special case, $GDW(M)$ contains $R$-quadratic metrics \cite{Rq}. In the following, it is shown that, it contains $PR$-quadratic metrics, too.
\subsection{Douglas curvature of Finsler metrics}
In this section, the Douglas curvature of Finsler metrics are considered. A novel category of Finsler metrics named $\bar{D}$-metrics is introduced by incorporating Douglas curvature, encompassing all Douglas metrics. The class of Douglas metrics encompasses the class of projectively flat Finsler metrics. However, not all Douglas metrics are locally projectively flat, as illustrated by the subsequent instance. In simpler terms, the statement \ref{LPFinD} applies.
\begin{ex}\label{D not PF} \cite{Yang}
Define Riemannian metric $\alpha$ and $1$-form $\beta$ on manifold $M$, by
\[
\widetilde{\alpha}=\eta^{\frac{m}{1-m}}\alpha, \quad \quad \widetilde{\beta}=\eta^{-1}\beta,
\]
for some $\eta=\eta(x)$ and $\widetilde{\beta}$ is parallel with respect to $\widetilde{\alpha}$ where $\widetilde{\alpha}$ and $\widetilde{\beta}$
\[
\widetilde{\alpha}=\sqrt{\frac{|y|^2}{|u|^2}},\quad \quad \widetilde{\beta}=\frac{<x,y>}{|u|^2},
\]
and $u=(u^1(x), ...,u^n(x))$ is a vector satisfying the following
\[
u^i=-2(\lambda+t<f,x>)x^i+t|x|^2 f^i+f^i,
\]
where $t$ is a constant and $f$ is a constant vector satisfying $tf\neq 0$ and $\lambda^2+t|f|^2\neq 0$. Then the m-Kropina metric $F=\alpha^m\beta^{1-m}$ is Douglasian but not locally projectively flat, where $m\neq 0, 1$.

\end{ex}
Every Douglas metric falls into the class of $GDW$-metrics but there are many particular examples that showcases the $GDW$-metrics which does not fit the criteria of being of Douglas type. In essence, the claim \ref{DinGDW} remains true.
\begin{ex} \cite{Osaka}
Put
\[
\Omega=\{(x,y,z) \in R^3 | x^2+y^2+z^2 <1\}, \quad p=(x,y,z) \in \Omega, \quad y=(u,v,w) \in T_p\Omega.
\]
Define the Randers metric $F=\alpha+\beta$ by
\[
\alpha=\frac{\sqrt{(-yu+xv)^2+(u^2+v^2+w^2)(1-x^2-y^2)}}{1-x^2-y^2}, \quad \beta=\frac{-yu+xv}{1-x^2-y^2}.
\]
The above Randers metric has vanishing flag curvature $K=0$ and $S$-curvature $\mathbf{S}=0$. $F$ has zero Weyl curvature then $F$ is of $GDW$ metric. But $\beta$  is not closed then $F$ is not of Douglas type.
\end{ex}
The above example is a $\bar{D}$-metric while it is not a Douglas metric. Because based on \eqref{D2} one has
\[
{D}_j{^i}_{kl|m}-{D}_j{^i}_{km|l}= {B}_j{^i}_{kl|m}-{B}_j{^i}_{km|l}=R_j{^i}_{ml.k}=0.
\]
In fact, every R-quadratic Finsler metric with vanishing $S$-curvature which is not of Douglas type is a non-trivial $\bar{D}$-metric. Although, we can not find these metrics between regular $(\alpha, \beta)$-metric of non-Randers type. Because of the main theorem in \cite{GDWTa} saying that; Let $F$ be a regular $(\alpha, \beta)$-metric of non-Randers type. Then $F$ is a generalized Douglas-Weyl $(GDW)$ metric with vanishing $S$-curvature if and only if it is a Berwald metric, or Theorem 2 in \cite{TaSaPe}\\
The following example presents a $\bar{D}$-metric which is not of Douglas type, too.
\begin{ex} \cite{HuMo}
Consider the following Randers metric defined nearby the origin
\[
F=\frac{\sqrt{|y|^2-(|xQ|^2|y|^2-<y,xQ>^2)}}{1-|xQ|^2}-\frac{<y, xQ>}{1-|xQ|^2},
\]
where $Q=\Big( q{^i}_j\Big)$ is an anti-symmetric matrix. $R{^i}_k = 0$ for $F$ but it is not a Berwald metric when $Q\neq 0$. We have
\[
\frac{\partial b_i}{\partial x^j}=-\frac{q_{ji}}{1-|xQ|^2}-2b_ib_j,
\]
which means that $\beta(x,y)=-\frac{<y, xQ>}{1-|xQ|^2}$ is not closed and then $F$ is not Douglas metric. Moreover, $e_{ij}=0$ which by Lemma 3.1 in \cite{Osaka} one finds that $\mathbf{S}=0$. Then $D_j{^i}_{kl|m}y^m=0$ and $F$ is a non-trivial $\bar{D}$-metric.
\end{ex}
In the following, a $GDW$-metric is presented which is not a $\bar{D}$-metric. In fact, every Finsler metric of (non-zero) constant flag curvature $\lambda$ with vanishing $S$-curvature is a $GDW$-metric which is not $\bar{D}$-metric (Corollary \ref{prop GDW not PRq}).
\begin{ex}\cite{BS}
The family of Randers metrics on $S^3$ constructed by Bao and Shen consists of non-Berwaldian weakly Berwald metrics. Denote generic tangent vectors on $S^3$ as
\[
u\frac{\partial }{\partial x}+ v\frac{\partial }{\partial y}+ z\frac{\partial }{\partial z}.
\]
The Finsler function for Bao-Shen's Randers space is given by
\[
F(x, y, z; u, v, w)=\alpha(x, y, z; u, v, w) + \beta(x, y, z; u, v, w),
\]
with
\[
\alpha=\frac{\sqrt{\lambda(cu-zv+yw)^2+(zu+cv-xw)^2+(-yu+xv+cw)^2}}{1+x^2+y^2+z^2},
\]
\[
\beta=\frac{\pm \sqrt{\lambda-1}(cu-zv+yw)}{1+x^2+y^2+z^2},
\]
where $\lambda > 1$ is a real constant. The above Randers metric has vanishing $S$-curvature and with positive constant flag curvature 1.
\end{ex}

\subsection{$PR$-quadratic Finsler metrics}
Here, $PR$-quadratic metrics as a subclass of $GDW(M)$ are considered.
\begin{thm} \label{Prop PRq=GDW}
Let $(M, F)$ be a Finsler space. $F$ is $PR$-quadratic if and only if
\be\label{PR}
D_j{^i}_{kl|0}=\frac{1}{n+1}(S_{.r}D_j{^i}_{kl})y^i,
\ee
where $S_{.r}$ denotes the differential of $S$-curvature $\mathbf{S}$ with respect to $y^r$.
\end{thm}
\begin{proof}
According to \eqref{PSpraylocal}, for Douglas curvature of $F$, one has
\be\label{h-derivGtild-G}
D_j{^i}_{kl||m}= D_j{^i}_{kl|m} + \frac{1}{n+1}\{S_{.j}D_m{^i}_{kl}+S_{.k}D_j{^i}_{ml}+S_{.l}D_j{^i}_{km}
\ee
\[
+ S_{.m}D_j{^i}_{kl}+SD_j{^i}_{kl.m} - S_{.r}D_j{^r}_{kl}\delta^i_m - S_{.r.m}D_j{^r}_{kl}y^i\},
\]
where $D_j{^i}_{kl||m}$ and $D_j{^i}_{kl|m}$ denote the horizontal derivative of $D_j{^i}_{kl}$ with respect to Berwald connection of $\widetilde{G}^i$ and $G^i$, respectively. Then one has
\be\label{D l-m}
D_j{^i}_{kl||m} - D_j{^i}_{km||l}= D_j{^i}_{kl|m} - D_j{^i}_{km|l}  - \frac{1}{n+1}\{S_{.r}D_j{^r}_{kl}\delta^i_m - S_{.r}D_j{^r}_{km}\delta^i_l
\ee
\[+(S_{.r.m}D_j{^r}_{kl}-S_{.r.l}D_j{^r}_{km})y^i\}.
\]
Based on  \eqref{PSpraylocal} and \eqref{D2} one could easily see that $PB_j{^i}_{kl}=D_j{^i}_{kl}$. It means that
\[
D_j{^i}_{kl||m} - D_j{^i}_{km||l}=PB_j{^i}_{kl||m} - PB_j{^i}_{km||l}.
\]
By applying the above equation and following Ricci identity \cite{Sh2}
\be\label{PRicci identity}
PB_j{^i}_{kl||m}-PB_j{^i}_{km||l}=PR_j{^i}_{ml.k},
\ee
in \eqref{D l-m}, one gets
\be\label{PR Dbar S-cur}
PR_j{^i}_{ml.k}= D_j{^i}_{kl|m} - D_j{^i}_{km|l}  - \frac{1}{n+1}\{S_{.r}D_j{^r}_{kl}\delta^i_m - S_{.r}D_j{^r}_{km}\delta^i_l
\ee
\[
+(S_{.r.m}D_j{^r}_{kl}-S_{.r.l}D_j{^r}_{km})y^i\}.
\]
Contracting the above equation by $y^m$ yields
\[
PR_j{^i}_{ml.k}y^m= D_j{^i}_{kl|0} - \frac{1}{n+1}\Big( S_{.r}D_j{^r}_{kl}\Big) y^i,
\]
which means that $PR_j{^i}_{lm.k}=0$ if and only if $D_j{^i}_{kl|0}=\frac{1}{n+1}\Big( S_{.r}D_j{^r}_{kl}\Big) y^i$.
\end{proof}
\begin{cor}
Every Douglas metric is a $PR$-quadratic Finsler metric.
\end{cor}
\begin{cor}
Every $PR$-quadratic Finsler metric is a $GDW$ metric.
\end{cor}
According to the above corollary, every $PR$-quadratic Finsler metric is a $GDW$ metric while as the following proposition shows that the class of $PR$-quadratic Finsler metrics is a proper of $GDW(M)$.
\begin{prop}\label{prop GDW not PRq}
Every non-Riemannian Finsler metric of constant flag curvature $\lambda \neq 0$ with vanishing $S$-curvature is a $GDW$-metric which is not $PR$-quadratic.
\end{prop}
\begin{proof}
$F$ is of constant curvature $\lambda$, then
\[
R{^i}_k=\lambda \{F^2\delta{i}_k-y_ky^i\}.
\]
Then by \eqref{Rikl} one has
\be\label{constant Rjikl}
R_j{^i}_{ml.k}=2\lambda (C_{jkl}\delta{^i}_m-C_{jkm}\delta{^i}_l).
\ee
On the other hands, $F$ has vanishing $S$-curvature then by \eqref{ES} and \eqref{D2} one gets $D_j{^i}_{kl}=B_j{^i}_{kl}$.
According to the following Ricci identity \cite{Sh2}
\be\label{Ricci identity}
B_j{^i}_{kl||m}-B_j{^i}_{km||l}=R_j{^i}_{lm.k},
\ee
\eqref{D2} and \eqref{PR Dbar S-cur} one finds that
\[
PR_j{^i}_{lm.k}= D_j{^i}_{kl|m} - D_j{^i}_{km|l}=B_j{^i}_{kl||m}-B_j{^i}_{km||l}=R_j{^i}_{lm.k}.
\]
The above equation and \eqref{constant Rjikl} one gets
\[
PR_j{^i}_{lm.k}= D_j{^i}_{kl|m} - D_j{^i}_{km|l}=2\lambda (C_{jkl}\delta{^i}_m-C_{jkm}\delta{^i}_l).
\]
Here, $F$ is not Riemannian and $\lambda \neq 0$, then the above equation means that $F$ is not $PR$-quadratic. But by its contracting with $y^m$ we have
\[
D_j{^i}_{kl|m}y^m =2\lambda C_{jkl} y^i,
\]
which means $F$ is a $GDW$-metric but not a $\bar{D}$-metric.
\end{proof}
\begin{cor}
Every non-Riemannian Finsler metric of constant flag curvature $\lambda \neq 0$ with vanishing $S$-curvature is a $GDW$-metric which is $\bar{D}$-metric.
\end{cor}
It is well-known that every $R$-quadratic Finsler metrics is a $GDW$-metric \cite{GDWTa}. But this rich class of Finsler metrics does not contain all Douglas metrics, while the class of $PR$-quadratic metrics does contain them. Then the calsses of $PR$-quadratic and $R$-quadratic metrics don't coincide with each other. Then the following Theorem would be interesting, as well.
\begin{thm}\label{thm PR=R}
A $R$-quadratic Finsler metric $(M,F)$ is $PR$-quadratic if and only if
\be\label{PRq=Rq}
S_{.j.k|m} = S_{.r}D_j{^r}_{km}
\ee
\end{thm}
\begin{proof}
For Finsler metric $(M,F)$, according to \eqref{D2} and Ricci identity \eqref{Ricci identity}, one has
\[
D_j{^i}_{kl|m} - D_j{^i}_{km|l}=R_j{^i}_{ml.k} - \frac{1}{n+1}\big(S_{.j.k|m}\delta{^i}_l-S_{.j.k|l}\delta{^i}_m
\]
\[
+ (S_{.j.l|m}-S_{.j.m|l})\delta{^i}_k + (S_{.k.l|m}-S_{.k.m|l})\delta{^i}_j\Big).
\]
The above equation and \eqref{PR Dbar S-cur} yield
\[
PR_j{^i}_{ml.k} = R_j{^i}_{ml.k} - \frac{1}{n+1}\Big([S_{.j.k|m}- S_{.r}D_j{^r}_{km}]\delta{^i}_l-[S_{.j.k|l}-S_{.r}D_j{^r}_{kl}]\delta{^i}_m + [S_{.j.l|m}-S_{.j.m|l}]\delta{^i}_k
\]
\[ + [S_{.k.l|m}-S_{.k.m|l}]\delta{^i}_j+[S_{.r.m}D_j{^r}_{kl}-S_{.r.l}D_j{^r}_{km}]y^i\Big),
\]
which means a $R$-quadratic Finsler metric $F$ is $PR$-quadratic if and only if \eqref{PRq=Rq} holds.
\end{proof}
\begin{cor}
Every Finsler metric with vanishing $S$-curvature is $PR$-quadratic if and only if it is $R$-quadratic.
\end{cor}
The class of Douglas metrics ($D(M)$) is an important class of Finsler metrics which is a proper subset of $GDW(M)$. In some special Finsler spaces, the class of $GDW(M)$ reduces to Douglas spaces $D(M)$ \cite{GDW=D Sph.}. In the following, we consider this case.
\begin{thm}
Let $(M,F)$ be a Finsler space which $GDW(M)=D(M)$. Then
\begin{itemize}
  \item []{(1)} $W(M) \subseteq D(M)$,
  \item [] {(2)} $PRq(M) = D(M)$,
  \item [] {(3)} $\bar{D}(M) = D(M)$.
\end{itemize}
Here $W(M)$, $D(M)$, $PRq(M)$ and $\bar{D}(M)$ denote the classes of Weyl, Douglas, $PR$-quadratic and $\bar{D}$ metrics, respectively.
\end{thm}
\begin{proof}
(1) Let $(M,F)$ be a Weyl metric and every Weyl metric is a $GDW$-metric. Then by assumption, it is a Douglas metric. \\
\\
(2) Now, assume that $F$ be a $PR$-quadratic Finsler metric. According to Proposition \ref{Prop PRq=GDW}, it is of $GDW$ type and then by the assumption, it is of Douglas type.\\
\\
(3) Finally, suppose that $F$ be a $\bar{D}$-metric, $D_j{^i}_{kl|m}-D_j{^i}_{km|l}=0$, then again it is of $GDW$ type and then by the assumption, it is of Douglas type.
\end{proof}

\end{document}